\documentclass{amsart}
\usepackage{amsmath, amssymb,epic,graphicx,mathrsfs,enumerate}
\usepackage[all]{xy}

\usepackage{amsthm}
\usepackage{amssymb}
\usepackage{latexsym}
\usepackage{longtable}
\usepackage{epsfig}
\usepackage{amsmath}
\usepackage{hhline}

 \DeclareMathOperator{\perm}{Sym}
 \DeclareMathOperator{\soc}{soc}

 \DeclareMathOperator{\frat}{Frat}

\DeclareMathOperator{\core}{Core}
\DeclareMathOperator{\Ee}{E}

\DeclareMathOperator{\End}{End}

\newtheorem{thm}{Theorem}
\newtheorem{cor}[thm]{Corollary}
 \newtheorem{lemma}[thm]{Lemma}
\newtheorem{prop}[thm]{Proposition} 
 \newtheorem{defn}[thm]{Definition}

\numberwithin{equation}{section}

\renewcommand{\footnote}{\endnote}
\newcommand{\ignore}[1]{}\makeglossary

\begin{document}
	\bibliographystyle{amsplain}
	\subjclass{20P05}
	\keywords{groups generation; waiting time; Sylow subgroups; permutations groups}
	\title[The expected number of elements to generate a group]{The expected number of elements to generate a finite  group
	with $d$-generated Sylow subgroups}

\author{Andrea Lucchini}
\address{
	Andrea Lucchini\\ Universit\`a degli Studi di Padova\\  Dipartimento di Matematica \lq\lq Tullio Levi-Civita\rq\rq \\email: lucchini@math.unipd.it}
\author{Mariapia Moscatiello}
\address{
	Mariapia Moscatiello\\ Universit\`a degli Studi di Padova\\  Dipartimento di Matematica \lq\lq Tullio Levi-Civita\rq\rq\\email: mariapia.moscatiello@gmail.com}	
		
	\begin{abstract}Given a finite group $G,$ let $e(G)$ be expected number of elements of  $G$ which have to be drawn at random, with replacement,
		before a set of generators is found.  
		If all the Sylow subgroups of $G$ can be generated by $d$ elements, then $e(G)\leq d+\kappa$ with $\kappa \sim 2.75239495.$ The number $\kappa$ is explicitly described in terms of
	the Riemann zeta function and is best possible. 	If $G$ is a permutation group of degree $n,$ then either $G=\perm(3)$ and $e(G)=2.9$ or $e(G)\leq  \lfloor n/2\rfloor+\kappa^*$ with $\kappa^* \sim 1.606695.$ These results improved the weaker ones obtained in \cite{alex}.
	\end{abstract}
	\maketitle
	
	\section{Introduction}
	
	In 1989, R. Guralnick \cite{rg} and the first author \cite{al} independently proved
	that if all the Sylow subgroups of a finite group $G$ can be generated by $d$ elements, then the group $G$ itself can be generated by $d+1$ elements.
	A probabilistic version of this result was obtained in \cite{alex}.
	Let $G$ be a nontrivial finite group and let $x=(x_n)_{n\in\mathbb N}$ be a sequence of independent, uniformly distributed $G$-valued random variables.
	We may define a random variable $\tau_G$ by
	$\tau_G=\min \{n \geq 1 \mid \langle x_1,\dots,x_n \rangle = G\}.$
	We denote by $e(G)$ the expectation $\Ee(\tau_G)$ of this random variable:
 $e(G)$ is the expected number of elements of $G$ which have to be drawn at random, with replacement,
	before a set of generators is found. In \cite{alex} it was proved that
if all the Sylow subgroups of $G$ can be generated by $d$ elements, then $e(G)\leq d+\eta$ with $\eta \sim 2.875065.$  This bound is not too far from  being best possible. Indeed in \cite{pom}, Pomerance  proved that if $\Omega_d$ is the set of all the $d$-generated finite abelian groups, then 
 $$\sup_{G\in \Omega_d} e(G)=d+\sigma, \text { where } \sigma \sim 2.11846.$$ However the bound $e(G)\leq d+\eta$ is approximative, and one could be interest in finding a best possible estimation for $e(G).$ We give an exhaustive answer to this question, proving the following result. 
\begin{thm}\label{tuno}
Let $G$ be a finite group. If all the Sylow subgroups of $G$ can be generated by $d$ elements, then $e(G)\leq d+\kappa$ with $\kappa \sim 2.75239495.$ The number $\kappa$ is explicitly described in terms of
the Riemann zeta function and is best possible.
\end{thm}

This bound can be further improved under some additional assumptions on $G.$  For example we prove that  if all the Sylow subgroups of $G$ can be generated by $d$ elements and $G$ is not soluble,  then $e(G)\leq d+2.7501$ (Proposition \ref{mainn}). A stronger result holds if $|G|$ is odd.
\begin{thm}\label{teorema}
	Let $G$ be a finite group of odd order. If all the Sylow subgroups of $G$ can be generated by $d$ elements, then $e(G)\leq d+\tilde \kappa$ with $\tilde \kappa \sim
	2.148668.$
\end{thm}

If $G$ is a $p$-subgroup of $\perm(n),$ then $G$ can be generated by $\lfloor n/p\rfloor$ elements (see \cite{kp}), so Theorem \ref{tuno} has the following consequence: if $G$ is a permutation group of degree $n,$ then $e(G)\leq  \lfloor n/2\rfloor+\kappa.$ However this bound is not best possible and a better result can be obtained:

\begin{cor}
	If $G$ is a permutation group of degree $n,$ then either $G=\perm(3)$ and $e(G)=2.9$ or $e(G)\leq  \lfloor n/2\rfloor+\kappa^*$ with $\kappa^* \sim 1.606695.$ 
\end{cor} 
The number $\kappa^*$ is best possible. Let 	$m=\lfloor n/2\rfloor$ and set $G_n=\perm(2)^m$ if $m$ is even, $G_n=\perm(2)^{m-1}\times \perm(3)$ if $m$ is odd. If $n\geq 8,$ then $e(G_n)-m$ increase with $n$ and $\lim_{n\to \infty} e(G)-m=1.606695.$
\section{Preliminary results}

Let $G$ be a finite group and use the following notations:
\begin{itemize}
	\item For a given prime $p,$ $d_p(G)$ is the smallest cardinality of a generating set of a Sylow $p$-subgroup of $G.$
	\item For a given prime $p$ and a positive  integer $t,$ $\alpha_{p,t}(G)$ is the number of complemented  factors of  order $p^t$ in a chief series of $G.$
	\item For a given prime $p,$ $\alpha_p(G)=\sum_t \alpha_{p,t}(G)$ is the number of complemented factors of $p$-power order in a chief series of $G.$
	\item $\beta(G)$ is the number of nonabelian factors  in a chief series of $G.$
\end{itemize}

\begin{lemma}\label{stime}For every finite group $G,$ we have:
	\begin{enumerate} 
		\item $\alpha_p(G)\leq d_p(G).$
		\item $\alpha_2(G)+\beta(G)\leq d_2(G).$
		\item If $\beta(G)\neq 0$, then $\beta(G)\leq d_2(G)-1.$
		\item If $\alpha_{2,1}(G)=0,$ then $\alpha_2(G)+\beta(G)\leq d_2(G)-1.$
		\item If $\alpha_{p,1}(G)=0,$ then $\alpha_p(G)\leq d_p(G)-1.$
	\end{enumerate}
\end{lemma}
\begin{proof}
(1), (2) and (3) are proved in \cite[Lemma 4]{alex}. Now assume that no complemented chief factor of  $G$ has order 2 and let $r=\alpha_2(G)+\beta(G)$. There exists a sequence $X_r\leq Y_r\leq\dots \leq X_1\leq Y_1$ of normal subgroups of $G$ such that, for every $1\leq i\leq r,$ 
 $Y_i/X_i$ is a complemented chief factor of $G$ of even order. Notice that $\beta(G/Y_1)=\alpha_2(G/Y_1)=0,$ hence $G/Y_1$ is a finite soluble group all of whose complemented chief factors have odd order, but then $G/Y_1$ has odd order and consequently $d_2(G)=d_2(Y_1).$ Moreover, as in the proof of
 \cite[Lemma 4]{alex}, $d_2(Y_1)\geq d_2(Y_1/X_1)+r-1.$ Since $|Y_1/X_1|\neq 2$ and the Sylow 2-subgroups of a finite nonabelian simple cannot be cyclic \cite[10.1.9]{rob}, we deduce $d_2(Y_1/X_1)\geq 2$ and consequently $d_2(G)=d_2(Y_1)\geq r+1.$ This proves (4). The proof of (5) is similar.
\end{proof}	

Recall (see \cite[(1.1)]{alex} for more details) that
	\begin{equation}\label{inizi1}
	\begin{aligned}e(G)=\sum_{n\geq 0}(1-P_G(n))
	\end{aligned}
	\end{equation}
	where $$P_G(n) =
	\frac{|\{(g_1,\dots,g_n)\in G^n \mid \langle g_1,\dots,g_n\rangle=G\}|}{|G|^n}$$ is the probability that $n$ randomly chosen
	elements of $G$ generate $G.$
	Denote by $m_n(G)$ the number of index $n$ maximal subgroups of $G.$	We have (see \cite[11.6]{sub}):
	\begin{equation}\label{inizi2}1-P_G(k)\leq \sum_{n\geq 2}\frac{m_n(G)}{n^{k}}.
	\end{equation}
	
	Using the notations introduced in \cite[Section 2]{pak}, we say that a maximal subgroup $M$ of $G$ is of type A if $\soc(G/\core_G(M))$ is abelian, of type B otherwise, and we denote by $m^A_n(G)$ (respectively $m^B_n(G)$)  the number of maximal subgroups of $G$ of type A (respectively B) of
	index $n.$  Given $t\in \mathbb N$ and $p\in \pi(G),$  define $$\mu^*(G,t)=\sum_{k\geq t}\left(\sum_{n\geq 5} \frac{m_n^B(G)}{n^k}\right), \quad \mu_p(G,t)=\sum_{k\geq t}\left(\sum_{n\geq 1}\frac{m_{p^n}^A(G)}{p^{nk}}\right).$$
	
	\begin{lemma}\label{e123}
Let $t\in \mathbb N.$ Then $e(G)\leq t+\mu^*(G,t)+\sum_{p\in \pi(G)}\mu_p(G,t).$
	\end{lemma}
	\begin{proof}
By (\ref{inizi1}) and (\ref{inizi2}), $$e(G)\leq t+\sum_{n\geq t}(1-P_G(n))\leq t+\sum_{k\geq t}\left(\sum_{n\geq 2}\frac{m_n(G)}{n^k}\right). \ \qedhere$$
	\end{proof}
	
\begin{lemma}\label{mustar}Let $t\in \mathbb N$. 
If $\beta(G)=0,$ then $\mu^*(G,t)=0.$
If $t\geq \beta(G)+3,$ then $$\mu^*(G,t)\leq \frac{\beta(G)(\beta(G)+1)}{2\cdot 5^{t-4}}\cdot \frac{1}{4}.$$	
\end{lemma}	
\begin{proof}
It follows from \cite[Lemma 8]{alex} and its proof.
\end{proof}

\begin{lemma}\label{mup}For $t\in \mathbb N$ and $p\in \pi(G).$ If $\alpha_p(G)=0,$ then $\mu_p(G,t)=0.$ 
\begin{enumerate}
\item If $\alpha_2(G)\leq t-1$ and $\alpha_{2,u}(G)\leq t-2$ for every $u>1,$ then $$\mu_2(G,t)\leq \frac{1}{2^{t-\alpha_2(G)-1}}.$$
\item  Let $p$ be an odd prime. If $\alpha_p(G)\leq t-2$ then $$\mu_p(G,t)\leq \frac{1}{p^{t-\alpha_p(G)-2}}\frac{1}{(p-1)^2}.$$
\end{enumerate}		
\end{lemma}
\begin{proof}
	It follows from \cite[Lemma 7]{alex} and its proof.
\end{proof}

Let $G$ be a finite soluble group and let $\mathcal A$ be a set of representatives for the irreducible $G$-module that are $G$-isomorphic to some complemented chief factor of $G.$ For every $A\in\mathcal A,$ let  $\delta_A$ be the number of
complemented factors $G$-isomorphic to $A$ in a chief series of $G$,  $q_A=|\End_{G}(A)|$, $r_A=\dim_{\End_{G}(A)}(A),$ $\zeta_A=0$ if $A$ is a trivial $G$-module, $\zeta_A=1$ otherwise. Moreover, for every $l\in \mathbb N,$ let $Q_{A,l}(s)$ be the Dirichlet polynomial defined by
$$Q_{A,l}(s)=1-\frac{q_A^{l+r_A\cdot \zeta_A}}{q_A^{r_A\cdot s}}.$$
By \cite[Satz 1]{g2}, for every positive integer $k$ we have
\begin{equation}\label{gzpr}
P_G(k)=\prod_{A\in \mathcal A}\left(\prod_{0\leq l\leq \delta_A-1}Q_{A,l}(k)\right).
\end{equation}
For every prime $p$ dividing $|G|$, let $\mathcal A_p$ be the subset of $\mathcal A$ consisting of the irreducible $G$-modules  having order a power of $p$ and let $$P_{G,p}(k)=\prod_{A\in \mathcal A_p}\left(\prod_{0\leq l\leq \delta_A-1}Q_{A,l}(k)\right).$$
\begin{defn}For every prime $p$ and every positive integer $\alpha$ let
$$C_{p,\alpha}(s)=\prod_{0\leq i\leq \alpha-1}\left(1-\frac{p^i}{p^s}\right), \quad D_{p,\alpha}(s)=\prod_{1\leq i\leq \alpha}\left(1-\frac{p^i}{p^s}\right).$$
\end{defn}
\begin{lemma}\label{confronti}
Let $G$ be a finite soluble group and let $k$ be a positive integer.
\begin{enumerate}
\item If $d_p(G)\leq d$, then $P_{G,p}(k)\geq D_{p,d}(k)$.
\item If $p$ divides $|G/G^\prime|,$ then $P_{G,p}(k)\geq C_{p,d}(k)$.
\item If $\alpha_{p,1}(G)=0,$ then $P_{G,p}(k)\geq C_{p,d}(k)$.
\item If $d_2(G)\leq d$, then $P_{G,2}(k)\geq C_{2,d}(k)$.
\end{enumerate}
\end{lemma}

\begin{proof}
		Suppose that $\mathcal A_p=\{A_1,\dots ,A_t\}$ and let $q_i=q_{A_i},$ $r_i=r_{A_i},$ $\zeta_i=\zeta_{A_i}$ and $\delta_i=\delta_{A_i}.$	
	Recall that
		\begin{equation}\label{a}
		P_{G,p}(k)=\prod_{\substack{1 \leq i \leq t \\        
				0 \leq l \leq \delta_i-1}}
		Q_{A_i,l}(k).
		\end{equation}
			By Lemma \ref{stime},  $\delta_1+\delta_2+\dots +\delta_t=\alpha_p(G)\leq d_p(G)\leq d$, hence the number of factors $Q_{A_i,l}(k)$ in (\ref{a}) is at most $d$. 
		We order these factors in such a way that $Q_{A_i,u}(k)$ precedes $Q_{A_j,v}(k)$ if either $i < j$ or $i=j$ and $u < v$. Moreover we order the elements of $\mathcal A_p$ in such a way that $A_1$ is the trivial $G$-module if $p$ divides $|G/G^\prime|.$
		
\noindent 1)
	Since $D_{p,d}(k)=0$ if $k\leq d$,  we may take $k>d$. 
To show that $P_{G,p}(k)\geq D_{p,d}(k),$ it is sufficient to show that the $j$-th factor $Q_j(k)=Q_{A_i,l}(k)$ of $P_{G,p}(k)$ is greater than the $j$-th factor 
$$D_j(k)=1-\frac{p^j}{p^k}$$ of $D_{p,d}(k)$. 
	If $j\leq \delta_1$ then $Q_j(k)=Q_{A_1,l}(k)
$ with $l=j-1$.
	If $j>\delta_1$ then $Q_j(k)=
Q_{A_i,l}(k)$
	for some $i\in \{2,\dots, t\}$ and $l\in \{0,\dots, \delta_i-1\}$, thus $$j=\delta_1+\delta_2+\dots +\delta_{i-1}+l+1\geq l+2.$$ In any case,
	$$q_i^{r_i\zeta_i}q_i^{l}\leq q_i^{r_i(l+1)}\leq q_i^{r_ij}.$$
We have $q_i=p^{n_i}$ for some $n_i\in \mathbb N.$ Since $j\leq d<k$, we deduce that
	\begin{equation*}
	\frac{q_i^{r_i\zeta_i}q_i^{l}}{q_i^{r_ik}}\leq \frac{q_i^{r_ij}}{q_i^{r_ik}} =\left(\frac{{p^{j}}}{{p^k}}\right)^{r_in_i}\leq \frac{p^j}{p^k}.
	\end{equation*}	
But then	
	\begin{equation*}
Q_j(k)=	1-\frac{q_i^{r_i\zeta_i}q_i^{l}}{q_i^{r_ik}}\geq 1-\frac{p^j}{p^k}=D_j(k).
	\end{equation*}
	
	\noindent
2)  Since $C_{p,d}(k)=0$ if $k < d$,  we may take $k\geq d$. 
	To show that $P_{G,p}(k)\geq C_{p,d}(k),$ it is sufficient to show that the $j$-th factor $Q_j(k)=Q_{A_i,l}(k)$ of $P_{G,p}(k)$ is greater than the $j$-th factor 
	$$C_j(k)=1-\frac{p^{j-1}}{p^k}$$ of $C_{p,d}(k)$.
	 If $i=1,$ then, by the way in which we ordered the elements of $\mathcal A_p$, we have $Q_j(k)=C_j(k)$. Otherwise, as we have seen in the proof of (1), $l+2\leq j$ so $r_i\zeta_i+l \leq r_i+j-2\leq r_i(j-1)$. Since $j\leq d\leq k$, we deduce that
	\begin{equation*}
	\frac{q_i^{r_i\zeta_i}q_i^{l}}{q_i^{r_ik}}\leq \frac{q_i^{r_i(j-1)}}{q_i^{r_ik}} \leq \frac{p^{j-1}}{p^k}\text {\ \ and\  \ }
	Q_j(k)=	1-\frac{q_i^{r_i\zeta_i}q_i^{l}}{q_i^{r_ik}}\geq 1-\frac{p^{j-1}}{p^k}=C_j(k).
	\end{equation*}
	
	\noindent 3) 	 Assume that no complemented chief factor of  $G$ has order $p.$ By (5) of Lemma \ref{stime}, $\alpha_p(G)\leq d_p(G)-1\leq d-1.$ But then the factors $Q_{A_i,l}(k)$ in (\ref{a}) are at most $d-1$ and, arguing as in the proof of (1), we conclude $P_{G,p}(k)\geq D_{p,d-1}(k)\geq C_{p,d}(k)$

	\noindent 4) We may assume $\alpha_2(G)\neq 0$ (otherwise $P_{G,2}(k)=1).$ Since $\alpha_{2,1}(G)\neq 0$ if and only if $2$ divides $|G/G^\prime|$, the conclusion follows from (2) and (3).
\end{proof}

\section{The main result}\label{main}
 \begin{prop}\label{mainn}Let $G$ be a finite group. If all the Sylow subgroups of $G$ can be generated by $d$ elements and $G$ is not soluble,  then $$e(G)\leq d+\kappa^* \quad \text { with } \quad \kappa^* \leq 2.7501.$$
 \end{prop}
 \begin{proof}
 Let $\beta=\beta(G).$ Since $G$ is not soluble, $\beta > 0$, hence by (2) and (3) of Lemma \ref{stime}, we have $1\leq \beta\leq d_2(G)-1\leq d-1$ and $\alpha_2(G)\leq d_2(G)-\beta\leq d-1.$ 
 	We distinguish two cases:	
 	
 	\noindent a) $\beta < d-1.$ By Lemma \ref{e123},  \ref{mustar} and
\ref{mup} and using an accurate estimation of $\sum_p (p-1)^{-2}$ given in \cite{hc}, we conclude
$$\begin{aligned}e(G)&\leq d+2+\mu^*(G,d+2)+\mu_2(G,d+2)+\sum_{p>2}\mu_p(G,d+2)\\&\leq d+2+\frac{1}{20}+\frac{1}{4}+\sum_{p>2}\frac{1}{(p-1)^2}\leq d+2.6751.
\end{aligned}$$
	\noindent b) $\beta = d-1.$ By  (2) and (4) of Lemma \ref{stime}, either $\alpha_2(G)=0$  or $\alpha_2(G)=\alpha_{2,1}(G)=1.$ In the first case $\mu_2(G,d+2)=0,$ in the second case $m_{2}^A(G)=1$ and consequently  $$\mu_{2}(G,d+2)=\sum_{k\geq d+2}\frac{m_{2}^A(G)}{2^{k}}\leq  \sum_{k\geq d+2}\frac{1}{2^{k}}\leq  \sum_{k\geq 4}\frac{1}{2^{k}}\leq \frac{1}{8}.$$
 By Lemma \ref{e123}, \ref{mustar} and
 \ref{mup}, we conclude
 $$\begin{aligned}e(G)&\leq d+2+\mu^*(G,d+2)+\mu_2(G,d+2)+\sum_{p>2}\mu_p(G,d+2)\\&\leq d+2+\frac{1}{4}+\frac{1}{8}+\sum_{p>2}\frac{1}{(p-1)^2}\leq d+2.7501.\quad\qedhere
 \end{aligned}$$
 \end{proof}
 
The previous proposition reduces the proof of Theorem \ref{tuno} to the particular case when $G$ is soluble. To deal with this case, we are going to introduce, for every positive integer $d$ and every set of primes $\pi$, a supersoluble group $H_{\pi,d}$ with the property that $e(G)\leq e(H_{\pi,d})$ whenever $G$ is soluble, $\pi(G)\subseteq \pi$ and the Sylow subgroups of $G$ are $d$-generated.

\begin{defn}Let $\pi$ be a finite set of prime integers with $2\in \pi,$ and let 
	$d$ be a positive integer. We define $H_{\pi,d}$ as the semidirect product
	$$H_{\pi,d}=\left(\left(\prod_{p \in \pi \setminus \{2\}} C_p^d\right) \rtimes C_2\right)\times C_2^{d-1}$$ where $C_p$ is the cyclic group of order $p$ and $C_2=\langle y \rangle$ acts on $A=\prod_{p \in \pi \setminus \{2\}} C_p^d$ by setting $x^y=x^{-1}$ for all $x\in A$.
	\end{defn} 
	
	 \begin{thm}\label{mainsol}Let $G$ be a finite soluble group. If all the Sylow subgroups of $G$ can be generated by $d$ elements,  then $e(G)\leq e(H_{\pi,d})$, where $\pi=\pi(G)\cup\{2\}.$
	 \end{thm}
\begin{proof}
Let $H=H_{\pi,d}$, $p\in \pi$ and $k\in\mathbb N.$ By (\ref{gzpr}), $P_{H,p}(k)=D_{p,d}(k)$ if $p\neq 2,$ while $P_{H,2}(k)=C_{2,d}(k).$ By Lemma \ref{confronti},
$P_{G,p}(k)\geq P_{H,p}(k)$ for every $p\in \pi(G).$ This implies 
$$P_G(k)=\prod_{p\in \pi(G)}P_{G,p}(k)\geq \prod_{p\in \pi}P_{H,p}(k)=P_H(G)$$
and consequently $e(G)=\sum_{k\geq 0}(1-P_G(k))\leq \sum_{k\geq 0}(1-P_H(k))=e(H).$
\end{proof} 
\begin{defn}Let $e_d=\sup_{\pi}e(H_{\pi,d})$ and $\kappa=\sup_d (e_d-d).$
\end{defn}
Let $2\in \pi$ and let $\pi^*=\pi\setminus \{2\}.$ Since $P_{H_{\pi,d}}(k)=0$ for all $k\leq d$ we have
$$\begin{aligned}
e(H_{\pi,d})&=\sum_{k\geq 0} \left(1 - P_{H_{\pi,d}}(k)\right)=d+1+\sum_{k\geq d+1}\left(1-C_{2,d}(k)\prod_{p\in \pi^*}D_{p,d}(k)\right)
\\&=d+1+\sum_{k\geq d+1}\left(1-\prod_{1\leq i\leq d}\left(1-\frac{2^{i-1}}{2^k}\right)\prod_{p\in \pi^*}\prod_{1\leq i\leq d}\left(1-\frac{p^i}{p^k}\right)\right)\\&=d+1+\sum_{t\geq 0}\left(1-\prod_{1\leq i\leq d}\left(1-\frac{2^{i-1}}{2^{t+(d+1)}}\right)\prod_{p\in \pi^*}\prod_{1\leq i\leq d}\left(1-\frac{p^i}{p^{t+(d+1)}}\right)\right)
.\end{aligned}$$
We immediately deduce that $e(H_{\pi,d})-d$ increase as $d$ increase. Moreover we have
$$\begin{aligned}
e_d-d&=\sup_\pi \left(e(H_{\pi,d})-d\right)\\&=1+\sum_{k\geq d+1}\left(1-\frac{(1-\frac{1}{2^k})}{(1-\frac{2^d}{2^k})}\prod_{p}\prod_{1\leq i\leq d}\left(1-\frac{p^i}{p^k}\right)\right).\end{aligned}$$
For $k=d+1$ the double product goes to $0$ while for $k\geq d+2$ goes to $\prod_{1\leq i\leq d}{\zeta(k-i)}^{-1}$ and so we get
$$\begin{aligned}
e_d-d&=2+\sum_{k\geq d+2}\left(1-\frac{(1-\frac{1}{2^k})}{(1-\frac{2^d}{2^k})}\prod_{1\leq i\leq d}{\zeta(k-i)}^{-1}\right)\\
&=2+\sum_{j\geq 1}\left(1-\frac{(1-\frac{1}{2^{j+(d+1)}})}{(1-\frac{1}{2^{j+1}})}\prod_{1\leq l\leq d}{\zeta(j+l)}^{-1}\right)\\
&=2+\sum_{j\geq 1}\left(1-\left(\frac{2^{j+1}-2^{-d}}{2^{j+1}-1}\right)\prod_{1+j\leq n\leq d+j}{\zeta(n)}^{-1}\right).
\end{aligned}$$
Let $c=\prod_{2\leq n\leq \infty}{\zeta(n)}^{-1}
.$ Since $e_d-d$ increases as $d$ grows,
we get
$$\begin{aligned}
\kappa& = \lim_{d \to \infty} e_d-d	\\	
&=2+\left(1-\left(\frac{2^{2}}{2^{2}-1}\right)c\right)+\sum_{j\geq 2}\left(1-\left(\frac{2^{j+1}}{2^{j+1}-1}\right)c\prod_{2\leq n\leq j}{\zeta(n)}\right)\\
&=2+\left(1-\frac{4}{3}\cdot c\right)+\sum_{j\geq 2}\left(1-\left(1+\frac{1}{2^{j+1}-1}\right)c\prod_{2\leq n\leq j}{\zeta(n)}\right).
\end{aligned}$$
Using the computer algebra system \textbf{PARI/GP} \cite{PARI2}, we get
\begin{equation*}
\kappa=2+\left(1-\frac{4}{3}\cdot c\right)+\sum_{j\geq 2}\left(1-\left(1+\frac{1}{2^{j+1}-1}\right)c\prod_{2\leq n\leq j}{\zeta(n)}\right)\sim 2.75239495.
\end{equation*}
Combining this result with Proposition \ref{mainn} and Theorem \ref{mainsol}, we obtain the proof of Theorem \ref{tuno}.

\section{Finite groups of odd order}

\begin{thm} \label{4}
	Let $G$ be a finite soluble group. There exists a finite supersoluble group $H$ such that
	\begin{enumerate}
		\item $\pi(H)=\pi(G)$,
		\item $P_G(k)\geq P_H(k)$ for all $k\in \mathbb{N}$,
		\item $d_p(G)\geq d_p(H)$ for all $p\in \pi(G)$,
		\item $\pi(G/G^\prime) \subseteq \pi(H/H^\prime).$
	\end{enumerate}
\end{thm}

\begin{proof} 	Let $\pi(G)=\{p_1,\dots ,p_n\}$ with $p_1\leq\dots\leq p_n$. For $i\in\{1,\dots,n\},$ set $\pi_i= \{p_1,\dots ,p_i\}$. We will prove, by induction on $i,$ that for every $i\in\{1,\dots,n\}$ there exists  a supersoluble group $H_i$ such that $\pi(H_i)=\pi_i$ and, for every $j\leq i$, 
	\begin{enumerate} \item $P_{H_i,p_j}(k)\leq P_{G,p_j}(k)$ for all $k\in \mathbb{N}$, \item $d_{p_j}(H_i)\leq d_{p_j}(G),$ \item if $C_{p_j}$ is an epimorphic image of $G$, then $C_{p_j}$ is  an epimorphic image of $H_{i}$.
		\end{enumerate}
Assume that $H_i$ has been constructed and set $p_{i+1}=p$ and $d_p(G)=d_p$.
We distinguish two different cases:

\noindent 1) Either $p$ divides $|G/G^\prime|$ or $G$ contains no complemented chief factor of order $p.$
	We consider the direct product
$H_{i+1}=H_i\times C_p^{d_p}.$ Clearly $P_{H_{i+1},p_j}(k)=P_{H_i,p_j}(k)\leq P_{G,p_j}(k)$ if $j\leq i.$ Moreover, by (2) and (3) of Lemma \ref{confronti}, $P_{H_{i+1},p}(k)=C_{p,d_p}(k)\leq P_{G,p}(k).$

\noindent 2) $p$ does not divide $|G/G^\prime|$ but $G$ contains a complemented chief factor which is isomorphic to a nontrivial $G$-module, say $A,$ of order $p$. In this case $G/C_G(A)$ is a nontrivial cyclic group whose order divides $p-1.$ Let $q$ be a prime divisor of
$|G/C_G(A)|$ (it must be $q=p_j$ for some $j\leq i$). Since $q$ divides $|G/G^\prime|,$ we have that $q$ divides also $|H_i/H_i^\prime|,$ hence there exists a normal subgroup $N$ of $H_i$ with $H_i/N\cong C_q$ and a nontrivial action of $H_i$ on $C_p$ with kernel $N.$
We use this action to construct the supersoluble group $H_{i+1}=C_p^{d_p}\rtimes H_i.$ Clearly $P_{H_{i+1},p_j}(k)=P_{H_i,p_j}(k)\leq P_{G,p_j}(k)$ if $j\leq i.$ Moreover, by (1) of Lemma \ref{confronti}, $P_{H_{i+1},p}(k)=D_{p,d_p}(k)\leq P_{G,p}(k).$

We conclude the proof, noticing that $H=H_n$ satisfies the requests in our statement.
\end{proof}
\begin{proof}[Proof of Theorem \ref{mainsol}]
Let $\pi(G)=\pi.$ By Theorem \ref{4}, there exists a supersoluble group $H$ such that $\pi(H)=\pi,$ $d_p(H)\leq d$ for every $p\in \pi$ and $P_G(k)\geq P_H(k)$ for every $k\in \mathbb N.$ In particular $e(G)=\sum_{k\geq 0}\left(1-P_G(k)\right)\leq \sum_{k\geq 0}\left(1-P_H(k)\right)=e(H).$ 

Since $H$ is supersoluble, if $A$ is $H$-isomorphic to a chief factor of $H,$ then $|A|=p$ for some $p\in \pi$ and $H/C_H(A)$ is a cyclic group of order dividing $p-1.$ If $p$ is a Fermat prime, then $H/C_H(A)$ is a 2-group and, since $|H|$ is odd, we must have $H=C_H(A).$ This implies that if $p\in \pi$ is a Fermat prime, then
$P_{H,p}(k)=C_{p,d_p(H)}(k)\geq C_{p,d}(k).$ For all the other primes in $\pi,$ by (1) of Lemma \ref{confronti} we have $P_{H,p}(k)\geq D_{p,d}(k).$
Therefore, denoting by $\Lambda$ the set of the Fermat primes and by $\Delta$ the set of the remaining odd primes, we get
\begin{equation*}
P_H(k)=\prod_{p\in \pi}P_{H,p}(k)
\geq \underset{p\in \Lambda}{\prod}C_{p,d}(k)\underset{p \in \Delta}{\prod}D_{p,d}(k).
\end{equation*}
It follows that
\begin{equation*}
	\begin{split}
	e(H)&=\sum_{k\geq 0}\left(1-P_H(k)\right)\\
	&\leq \sum_{k\geq 0}\left(1-\underset{p\in \Lambda}{\prod}\prod_{1\leq i\leq d}\left(1-\frac{p^{i-1}}{p^k}\right)\underset{\substack{p\in \Delta\\ p\neq 2}}{\prod}\prod_{1\leq i\leq d}\left(1-\frac{p^{i}}{p^k}\right)\right)\\
	&=d+1+\sum_{k\geq d+1}\left(1-\underset{p\in \Lambda}{\prod}\prod_{1\leq i\leq d}\left(1-\frac{p^{i-1}}{p^k}\right)\underset{\substack{p\in \Delta}}{\prod}\prod_{1\leq i\leq d}\left(1-\frac{p^{i}}{p^k}\right)\right)\\
	&= d+1+\sum_{t\geq 0}\left(1-\prod_{p\in \Lambda}\prod_{1\leq i\leq d}\left(1-\frac{p^{i-1}}{p^{t+(d+1)}}\right)\prod_{p\in \Delta
		}\prod_{1\leq i\leq d}\left(1-\frac{p^i}{p^{t+(d+1)}}\right)\right).
	\end{split}
	\end{equation*}
	Let $$\tilde\kappa_d=\sum_{t\geq 0}\left(1-\prod_{p\in \Lambda}\prod_{1\leq i\leq d}\left(1-\frac{p^{i-1}}{p^{t+(d+1)}}\right)\prod_{p\in \Delta
	}\prod_{1\leq i\leq d}\left(1-\frac{p^i}{p^{t+(d+1)}}\right)\right)+1.$$
It can be easily check that $\tilde\kappa_d$ increase as $d$ increases. Let
\begin{equation*}
b={\underset{1\leq n\leq \infty}{\prod}\left(1-\frac{1}{2^{n}}\right)^{-1}},\qquad c=\prod_{2\leq n\leq \infty}{\zeta(n)}^{-1}
\end{equation*}
and let $\Lambda^*=\{3,\;5,\;17,\ 257,\; 65537\}$ be the set of the known Fermat primes.
Similar computations to the ones in the final part of Section \ref {main}
lead to the conclusion
$$\begin{aligned}
\tilde\kappa_d
& \leq 3\!-\!\frac{b\cdot c}{2}\prod_{p\in \Lambda}\frac{p^{2}}{p^{2}\!-\!1}+\sum_{j\geq 2} \left(1\!-\! b\!\!\underset{1\leq n\leq j}{\prod}\!\!\!\left(\!1-\!\frac{1}{2^{n}}\right)\prod_{p\in \Lambda}\left(1\!+\!\frac{1}{p^{j+1}-1}\right)\!c\!\!\prod_{2\leq n\leq j}\!\!\!{\zeta(n)}\right) \\
& \leq 3\!-\!\frac{b\cdot c}{2}\!\prod_{p\in \Lambda^*}\!\frac{p^{2}}{p^{2}\!-\!1}\!+\sum_{j\geq 2} \left(\!1\!-\! b\!\!\underset{1\leq n\leq j}{\prod}\!\!\!\left(1-\frac{1}{2^{n}}\right)\prod_{p\in \Lambda^*}\!\!\left(1\!+\!\frac{1}{p^{j+1}-1}\right)\!c\!\!\!\prod_{2\leq n\leq j}\!\!\!{\zeta(n)}\right)\!.
\end{aligned}$$
Let $$\tilde \kappa=3\!-\!\frac{b\cdot c}{2}\!\prod_{p\in \Lambda^*}\!\frac{p^{2}}{p^{2}\!-\!1}\!+\sum_{j\geq 2} \left(\!1\!-\! b\underset{1\leq n\leq j}{\prod}\!\!\!\left(1-\frac{1}{2^{n}}\right)\prod_{p\in \Lambda^*}\!\!\left(1\!+\!\frac{1}{p^{j+1}-1}\right)\!c\!\!\!\prod_{2\leq n\leq j}\!\!\!{\zeta(n)}\right).$$
With the help of \textbf{PARI/GP}, we get  that $\tilde \kappa \sim 2.148668.$ \end{proof}

\section{Permutation groups}

\begin{thm}{\cite[Corollary]{kp}\label{kopr}} If $G$ is a $p$-subgroup of $\perm(n),$ then $G$ can be generated by $\lfloor n/p\rfloor$ elements.
\end{thm}


\begin{thm}{\cite[Theorem 10.0.5]{nina}}\label{nin}
	The chief length of a permutation group of degree $n$ is at most $n-1.$		
\end{thm}

\begin{lemma}\label{basso}
	If $G\leq \perm(n)$ and $n\geq 8$, then $\beta(G)\leq \lfloor n/2 \rfloor -3$.
\end{lemma}	
\begin{proof}
	Let $R(G)$ be the soluble radical of $G.$ By \cite[Theorem 2]{holt} $G/R(G)$ has a faithful permutation representation of degree at most $n,$ so we may assume $R(G)=1$. In particular $\soc(G)=S_1\times \cdots \times S_r$ where $S_1,\dots,S_r$ are nonabelian simple groups and, by \cite[Theorem 3.1]{ep}, $n\geq 5r.$ Let $K=N_G(S_1)\cap \dots \cap N_G(S_r).$ We have that $K/\soc(G)$ is soluble and that $G/K\leq \perm(r),$ so by Theorem \ref{nin}, $\beta(G/K)\leq r-1$ (and indeed $\beta(G/K)=0$ if $r\leq 4).$ But then $\beta(G)\leq 2r-1\leq 2\lfloor n/5 \rfloor -1$ if $r\geq 5,$ $\beta(G)\leq r \leq \lfloor n/5 \rfloor$ otherwise. 
\end{proof}	

\begin{lemma}\label{notsol}
	Suppose that $G\leq \perm(n)$ with $n\geq 8$. If $G$ is not soluble, then $$e(G)\leq \lfloor n/2 \rfloor +1.533823.$$
\end{lemma}
\begin{proof} Let $m=\lfloor n/2 \rfloor.$ By Theorem $\ref{kopr},$ $d_2(G)\leq m.$ Since $G$ is not soluble, we must have $\beta(G)\geq 1$. By Lemma \ref{basso}, $\beta(G)\leq m-3$, hence, by Lemma \ref{mustar},  $\mu^*(G,m)\leq 1/4.$ By (2) and (4) of  Lemma \ref{stime}, $\alpha_2(G)\leq m-1$ and $\alpha_{2,u}(G)\leq m-2$ for every $u > 1,$ hence, by Lemma \ref{mup}, $\mu_2(G,m)\leq 1.$ If $p\geq 5,$ then, by  Theorem $\ref{kopr},$ $m-\alpha_p(G)\geq m-d_p(G)\geq m - \lfloor n/5 \rfloor \geq 3$ so, 
	by Lemma \ref{mup},	
	$\mu_p(G,m)\leq (p(p-1)^2)^{-1}.$ Since $n\geq 8$ we have $m-\alpha_3(G)\geq m - \lfloor n/3 \rfloor \geq 2$
	if $n\neq 9.$ On the other hand, it can be easily checked that $\alpha_3(G)\leq 2$ for every unsoluble subgroup $G$ of $\perm(9),$ so
	$m-\alpha_3(G)\geq 2$ also when $n=9.$ But then, again by Lemma \ref{mup},
	$\mu_3(G,m)\leq 1/4.$ It follows
	$$\begin{aligned}e(G)&\leq m + \mu^*(G,m)+\mu_2(G,m)+\mu_3(G,m)+\sum_{p>3}\mu_p(G,m)\\&\leq m+\!\frac{1}{4}\!+\!1+\!\frac{1}{4}\!+\!\sum_{p\geq 5}\frac{1}{p(p-1)^2}\leq m\!+\!\frac{3}{2}\!+\!\sum_{n\geq 5}\frac{1}{n(n-1)^2}\leq m\!+\!1.533823.\ \qedhere \end{aligned}$$
\end{proof}

\begin{lemma}
	Suppose that $G\leq \perm(n)$ with $n\geq 8$. If $G$ is soluble
	and $\alpha_{2,1}(G)<\lfloor n/2 \rfloor,$  then 
	$$e(G)\leq \lfloor n/2 \rfloor + 1.533823.$$
\end{lemma}

\begin{proof}
	Let $\alpha=\alpha_{2,1}(G)$, $\alpha^*=\sum_{i>1}\alpha_{2,i}(G)$ and $m=\lfloor n/2 \rfloor.$ Notice that $\alpha^*\leq m-1$ by Lemma \ref{stime} (4).  Set $$\mu_{2,1}(G,t)=\sum_{k\geq t}\frac{m_{2}^A(G)}{2^{k}},\quad 
	\mu_{2,2}(G,t)=\sum_{k\geq t}\left(\sum_{n\geq 2}\frac{m_{2^n}^A(G)}{2^{nk}}\right).$$
	We distinguish two cases:	
	
	\noindent a) $\alpha_{2,u}(G)<m-1$ for every $u\geq 2.$ 
	Since ${m_{2}^A(G)}=2^\alpha-1,$ we have $$\mu_{2,1}(G,m)\leq \sum_{k\geq m}\frac{2^\alpha}{2^k}=\frac{1}{2^{m-\alpha-1}}\leq 1.$$ Moreover, arguing as in the proof of \cite[Lemma 7]{alex}, we deduce $$\mu_{2,2}(G,m)\leq \frac{1}{2^{m-\alpha^*-1}}\leq 1.$$ Notice that if $\alpha=m-1,$ then $\alpha^*\leq 1$ and consequently  $\mu_{2,2}(G,m)\leq 2^{2-m}\leq 1/4.$ Similarly, if $\alpha^*=m-1,$ then $\alpha\leq 1$ and   $\mu_{2,1}(G,m)\leq 2^{2-m}\leq 1/4.$ If follows that $\mu_2(G,m)=\mu_{2,1}(G,m)+\mu_{2,2}(G,m)\leq 5/4.$ Except in the case when $n=9$ and $\alpha_3(G)=3$, arguing as in the end of Lemma \ref{notsol}, we conclude
	$$\begin{aligned}e(G)&\leq m +\mu_2(G,m)+\mu_3(G,m)+\sum_{p>3}\mu_p(G,m)\\&\leq m+\frac{5}{4}+\frac{1}{4}+\sum_{p\geq 5}\frac{1}{p(p-1)^2}\leq m+1.533823.\quad  \end{aligned}$$
	We remain with the case when $G$ is a soluble subgroup of $\perm(9)$ with $\alpha_3(G)=3.$ This occurs only if $G$ is contained in the wreath product 
	$\perm(3)\wr \perm(3).$ In particular $\alpha_2(G)\leq 3.$ If $\alpha_2(G)\leq 2,$ then, by Lemma 	\ref{mup},
	$$e(G)\leq 5+\mu_2(G,5)+\mu_3(G,5)\leq 5+\frac{1}{4}+\frac{1}{4} = 5.5.$$
	We have $\alpha_2(G)=\alpha_3(G)=3$ only in two cases: $\perm(3)\times \perm(3) \times \perm 3$, $\langle (1,2,3), (4,5,6), (1,4)(2,5)(3,6), (1,2)(4,5)\rangle \times \perm(3).$ In this two cases, $G$ contains exactly 16 maximal subgroups, 7 with index 2 and 9 of index 3. But them
	$$e(G)\leq 4+\sum_{k\geq 4}\frac{m_2(G)}{2^k}+\sum_{k\geq 4}\frac{m_3(G)}{3^k}= 4+\sum_{k\geq 4}\frac{7}{2^k}+\sum_{k\geq 4}\frac{9}{3^k}=4+\frac{7}{8}+\frac{1}{6}\sim 5.0417. 
	$$

	\noindent b) $\alpha_{2,u}(G)=m-1$ for some $u\geq 2.$ In this case
	$m_2^A(G)\leq 1,$ so $$\mu_{2,1}(G,m+1)\leq \sum_{k\geq m+1}\frac{1}{2^{k}}=\frac{1}{2^m}\leq \frac{1}{16}.$$
	Moreover, by \cite[Lemma 5]{alex}, $m_{2^u}^A(G)\leq 2^{u\alpha_{2,t}(G)+u}$, hence
	$$\begin{aligned}\mu_{2,2}(G,m+1)&=\sum_{k\geq m+1}\left(\sum_{n\geq 2}\frac{m_{2^n}^A(G)}{2^{nk}}\right)=\sum_{k\geq m+1}\frac{m_{2^u}^A(G)}{2^{uk}}\leq \sum_{k\geq m+1}\frac{2^{u\alpha_{2,t}(G)+u}}{2^{uk}}\\&\leq \sum_{k\geq m+1}\frac{2^{um}}{2^{uk}}=\frac{1}{2^u-1}\leq \frac 1{3}.
	\end{aligned}$$
	If $p\geq 5,$ then $m-\alpha_p(G)\geq 3$  so, 
	by Lemma \ref{mup},	
	$\mu_p(G,m+1)\leq (p(p-1))^{-2}.$ Moreover $m-\alpha_3(G)\geq 2$ (notice that there is no subgroup of $\perm(9)$ with $\alpha_3(G)=3$ and $\alpha_{2,u}(G)=3$ for some $u\geq 2)$, so, again by Lemma \ref{mup}, $\mu_3(G,m+1)\leq 1/12.$ It follows
	$$\begin{aligned}e(G)&\leq m\! + \!1 +\mu_{2,1}(G,m+1)+\mu_{2,2}(G,m+1)+\mu_3(G,m+1)+\sum_{p>3}\mu_p(G,m+1)\\&\leq m\!+\!1\!+\!\frac{1}{16}\!+\!\frac 1{3}\!+\!\frac 1{12}\!+\!\!\sum_{p\geq 5}\frac{1}{p^2(p-1)^2}\leq \frac {71}{48}\!+\!\!\sum_{n\geq 5}\frac{1}{n^2(n-1)^2}\leq m\!+\!1.4843.\!\qedhere \end{aligned}$$
	\end{proof}

When $G\leq \perm(n)$ and $n \leq 7,$ the precise value of $e(G)$ can be computed by \textbf{GAP} \cite{gap} using  the formula $$e(G)=-\sum_{H < G}\frac{\mu_G(H)|G|}{|G|-|H|},$$ where  $\mu_G$ is the  M\"obius function defined on the
subgroup lattice  of $G$ (see \cite[Theorem 1]{mon}). The crucial information are summarized in the following lemma.

\begin{lemma}
	Suppose that $G\leq \perm(n)$ with $n\leq 7$. Either $e(G)\leq \lfloor n/2 \rfloor+1$ or one of the following cases occurs:\begin{enumerate}
		\item $G\cong \perm(3),$ $n=3$, $e(G)=29/10;$
		\item $G\cong C_2\times C_2,$ $n=4$, $e(G)=10/3;$
		\item $G\cong D_8,$ $n=4$, $e(G)=10/3;$
		\item $G\cong C_2\times \perm(3),$ $n=5$, $e(G)=1181/330;$
		\item $G\cong C_2\times C_2\times C_2,$ $n=6$, $e(G)=94/21;$
		\item $G\cong C_2\times D_8,$ $n=6$, $e(G)=94/21;$
			\item $G\cong C_2\times C_2\times \perm(3),$ $n=7$, $e(G)=241789/53130;$
			\item $G\cong D_8\times \perm(3),$ $n=7$, $e(G)=241789/53130.$
	\end{enumerate}
\end{lemma}

\begin{thm}
	Let $G$ be a permutation group of degree $n\neq 3$. If  $\alpha_{2,1}(G)=\lfloor n/2 \rfloor$, then
	$e(G)\leq \lfloor n/2 \rfloor+\nu,$ with $\nu\sim 1.606695.$
\end{thm}
\begin{proof}
Let $m=\lfloor n/2 \rfloor.$ We have that $\alpha_{2,1}(G)=m$ if and only if $C_2^m$ is an epimorphic image of $G.$ By \cite{kp} if $C_2^m$ is an epimorphic image of $G$ then $G$ is the direct product of its transitive constituents
and each constituent is one of the following: $\perm(2),$ of degree 2, $\perm(3),$ of degree 3,
$C_2 \times C_2,$ $D_8,$ of degree 4, and the central product $D_8 \circ D_8,$ of degree 8. Consequently:
$$
G/\frat(G)\simeq \begin{cases} C_2^m &\text{ if $n=2m$},\\
C_2^{m-1}\times \perm(3) &\text{ if $n=2m+1.$}
\end{cases}
$$
And so, by (\ref{gzpr}),
\begin{equation*}
P_G(k)=P_{G/\frat(G)}(k)=\prod_{0\leq i\leq m-1}\left(1-\frac{2^i}{2^k}\right){\left(1-\frac{3}{3^k}\right)}^{n-2m}.
\end{equation*}
Setting $\eta=0$ if $n$ is even,  $\eta=1$ otherwise,  we have
\begin{equation*}
\begin{split}
e(G)&=\sum_{k\geq 0}\left(1-P_G(k)\right)\leq \sum_{k\geq 0}\left(1-\prod_{0\leq i\leq m-1}\left(1-\frac{2^i}{2^k}\right){\left(1-\frac{3}{3^k}\right)^\eta}\right)\\
&= m+ \sum_{k\geq m}\left(1-\prod_{0\leq i\leq m-1}\left(1-\frac{2^i}{2^k}\right){\left(1-\frac{3}{3^k}\right)^\eta}\right)\\
&=  m+ \sum_{j\geq 0}\left(1-\prod_{1\leq l\leq m}\left(1-\frac{1}{2^{j+l}}\right){\left(1-\frac{3}{3^{j+m}}\right)^\eta}\right).
\end{split}
\end{equation*}
Set $$\omega_{m,\eta}=\sum_{j\geq 0}\left(1-\prod_{1\leq l\leq m}\left(1-\frac{1}{2^{j+l}}\right){\left(1-\frac{3}{3^{j+m}}\right)^\eta}\right).$$
Clearly $\omega_{m,0}$ increase with $m.$ On the other hand, if $m\geq 4$ and $j\geq 0$ then
\begin{equation*}
\begin{split}
\left(1-\frac{1}{2^{j+m+1}}\right)\left(1-\frac{3}{3^{j+m+1}}\right)\leq \left(1-\frac{3}{3^{j+m}}\right)
\end{split}
\end{equation*}
and so $\omega_{m,1} \leq \omega_{m+1,1}$ if $m \geq 4.$
Moreover $$\lim_{m\to \infty}\omega_{m,1}=\lim_{m\to \infty}\omega_{m,0}\sim 1.606695.$$
But then $e(G)\leq m+1.606695$ whenever $m
\geq 4.$ The value of $e(G)$ when $n$ is small is given by the following table (that indicates also how fast $e(G)-m$ tends to 1.606695).
\begin{center}
	\begin{tabular}{|p{0.7cm}|p{4.2cm}|}
		\hline
			\rule[-2mm]{0mm}{0,7cm}
		$n$ &  $e(G)$ \\
		\hline
		\rule[-2mm]{0mm}{0,7cm}
		$2$ &  $2$\\
		\hline
		\rule[-2mm]{0mm}{0,7cm}
		$3$ &  $\frac{29}{10}=2.900$\\
		\hline
		\rule[-2mm]{0mm}{0,7cm}
		$4$ &  $\frac{10}{3}\sim 3.334$\\
		\hline
		\rule[-2mm]{0mm}{0,7cm}
		$5$ &  $\frac{1181}{330}\sim 3.579$\\
		\hline
		\rule[-2mm]{0mm}{0,7cm}
		$6$ &  $\frac{94}{21}\sim 4.476$\\
		\hline
		\rule[-2mm]{0mm}{0,7cm}
		$7$ &  $\frac{241789}{53130}\sim 4.551$\\
		\hline
		\rule[-2mm]{0mm}{0,7cm}
		$8$ &  $\frac{194}{35}\sim 5.5429$\\
		\hline
	\end{tabular}
\begin{tabular}{|p{0.7cm}|p{4.2cm}|}
	\hline
	\rule[-2mm]{0mm}{0,7cm}
	$n$ &  $e(G)$ \\
	\hline
	\rule[-2mm]{0mm}{0,7cm}
	$9$ &  $\frac{4633553}{832370}\sim 5.5667$\\
	\hline
	\rule[-2mm]{0mm}{0,7cm}
	$10$ &  $\frac{7134}{1085}\sim 6.5751$\\
	\hline
	\rule[-2mm]{0mm}{0,7cm}
	$11$ &  $\frac{3227369181}{490265930}\sim 6.5828$\\
	\hline
	\rule[-2mm]{0mm}{0,7cm}
	$12$ &  $\frac{74126}{9765}\sim 7.59099$\\
	\hline
	\rule[-2mm]{0mm}{0,7cm}
	$13$ &  $\frac{6399598043131}{842767133670}\sim 7.59355$\\
	\hline
	\rule[-2mm]{0mm}{0,7cm}
	$14$ &  $\frac{10663922}{1240155}\sim 8.59886$\\
	\hline
	\rule[-2mm]{0mm}{0,7cm}
	$15$ &  $\frac{70505670417749503}{8198607229768494}
	\sim 8.59971$\\
	\hline
\end{tabular}	
\end{center}
From the information contained in this table, we deduce that $e(G)\leq m+1.606695$, except when $G=\perm(3).$
\end{proof}

\end{document}